\documentclass{siamart220329}

\makeatletter
\def\refstepcounter@optarg[#1]#2{\refstepcounter@noarg{#2}}
\makeatother

\usepackage{amssymb}
\usepackage{bm}
\usepackage{algorithm}
\usepackage{algpseudocode}
\usepackage{graphicx}
\graphicspath{{figures/}{./}}
\usepackage{cite}
\usepackage{xurl}
\usepackage{hyperref}

\def\eq#1{Eq.~\eqref{eq:#1}}
\def\eqs#1{Eqs.~\eqref{eq:#1}}
\def\eqn#1{\eqref{eq:#1}}
\def\figref#1{Fig.~\ref{fig:#1}}

\newcommand{\Real}{\mathbb{R}}
\newcommand{\bfell}{\boldsymbol{\ell}}  % boldsymbol?
\newcommand{\bfL}{\mathbf{L}}

\newcommand{\bfD}{\mathbf{D}}
\newcommand{\bfd}{\mathbf{d}}

\newcommand{\bfA}{\mathbf{A}}
\newcommand{\bfx}{\mathbf{x}}
\newcommand{\bfy}{\mathbf{y}}
\newcommand{\bfb}{\mathbf{b}}

\newcommand{\bfB}{\mathbf{B}}
\newcommand{\bfC}{\mathbf{C}}
\newcommand{\bfX}{\mathbf{X}}
\newcommand{\bfY}{\mathbf{Y}}
\newcommand{\bfalpha}{\boldsymbol{\alpha}}
\newcommand{\bfh}{\mathbf{h}}
\newcommand{\bfH}{\mathbf{H}}
\newcommand{\bfv}{\mathbf{v}}
\newcommand{\bfg}{\mathbf{g}}
\newcommand{\bfQ}{\mathbf{Q}}
\newcommand{\bfLambda}{\boldsymbol{\Lambda}}
\DeclareMathOperator{\sgn}{sgn}

\begin{document}

\headers{Hemiplex Cholesky factorization of symmetric matrices}{A.~Edelman and T.~E.~Holy}

\title{Jordan algebras, hemiplex numbers, and the Cholesky decomposition of arbitrary symmetric matrices\thanks{\funding{This work was supported by the National Institutes of Health R01DC020034 (TEH); the U.S. National Science Foundation (awards CNS-2346520, RISE-2425761, DMS-2325184), the Defense Advanced Research Projects Agency (Agreement HR00112490488), the Department of Energy National Nuclear Security Administration (award DE-NA0004266), and the U.S. Air Force Research Laboratory (Cooperative Agreement FA8750-19-2-1000) (all AE). Any opinions, findings, conclusions, or recommendations expressed are those of the authors and do not necessarily reflect the views of the U.S. Government or any agency thereof, which assumes no liability for the information contained herein. Reference to any specific commercial product, process, or service does not constitute endorsement by the U.S. Government.}}}

\author{Alan Edelman\thanks{Department of Mathematics, Massachusetts Institute of Technology, Cambridge, MA (\email{edelman@mit.edu}).}
\and Timothy E. Holy\thanks{Department of Neuroscience, Washington University School of Medicine, St.~Louis, MO (\email{holy@wustl.edu}). Corresponding author.}}

\maketitle

\begin{abstract}
Positive-semidefinite matrices are most efficiently factored using the Cholesky decomposition.  For indefinite matrices, the Cholesky factorization does not exist, and the alternatives face greater challenges in achieving numeric stability and preservation of banded structure.  Here we pursue an analogy between the requirement for positive-semidefinite matrices and the solution of the quadratic equation $x^2 = c$ for $c \le 0$.  It is shown that a non-associative algebra, 
called the hemiplex numbers, allows the Cholesky factorization to be computed for arbitrary symmetric matrices. Crucially, the hemiplex Cholesky factorization does not require pivoting for its existence or stability, allowing it to preserve banded structure. For singular matrices it produces a parametrization of the null space, and provides opportunity for truncation of nearly-null directions in a manner similar to common usage of the singular value decomposition.  The hemiplex Cholesky factorization may be a practically useful addition to the tools for solving symmetric linear equations.
\end{abstract}



\section{Introduction}

For linear systems
\begin{equation}
\bfA \bfx = \bfb,
\end{equation}
where $\bfA$ is a symmetric positive-definite matrix, the most widely-used method of solution is via the Cholesky decomposition $\bfA = \bfL\bfL^T$ where $\bfL$ is a lower triangular matrix.  This matrix factorization, reviewed in section~\ref{sec:Cholesky}, can be computed more efficiently than alternative decompositions such as LU, QR, or the singular value decomposition (SVD).  For sparse matrices, a major advantage of this triangular decomposition is its preservation of banded structure\cite{Martin1965}, making computation of the factorization much more efficient.

When $\bfA$ is not positive-semidefinite, the standard Cholesky factorization does not exist.  Two widely-used alternatives include $LDL^T$ factorization, which introduces a diagonal matrix $\bfD$ with elements of either sign, and the Bunch-Kaufman $\bfL\bfB\bfL^T$ factorization which allows block-diagonal $\bfB$ with blocks of size $1\times 1$ or $2\times 2$.  The Bunch-Kaufman factorization \cite{Bunch1977} exists for any symmetric matrix and, with suitable pivoting, has good stability properties \cite{Higham1997}.  However, this factorization does not preserve banded structure\cite{Bunch1977} (but see \cite{Irony2006,Kaufman2007}).
The $LDL^T$ factorization does preserve such structure, but as is well known (and recapped below) it does not exist for all matrices.  The numeric stability of $LDL^T$ decomposition has been the subject of much investigation (e.g., \cite{Gill1996,Saunders1996,Gould2007,Arie2009}).

There is an analogy between Cholesky factorization and the fact that the quadratic equation $x^2 = c$ lacks real-valued solutions for $c < 0$.  The desirability of ``solving'' this equation is motivation for introducing the complex numbers $a+b\imath$ with $\imath^2 = -1$.  It will be shown that the Cholesky factorization is amenable to similar treatment.

\section{The Cholesky and $LDL^T$ factorizations}
\label{sec:Cholesky}

For a symmetric $K\times K$ real-valued matrix
\begin{equation}
\bfA = \begin{bmatrix}
A_{11} & A_{12} & A_{13} & \ldots \\
A_{21} & & & \\
A_{31} & & \bfA_2 & \\
$\vdots$ & & & \\
\end{bmatrix},
\end{equation}
where $\bfA_2$ is a symmetric $(K-1)\times(K-1)$ matrix, one computes the Cholesky factorization by seeking a $K$-vector $\bfell$ for which $\bfA - \bfell \bfell^*$ ($\bfell^*$ denoting the adjoint of $\bfell$) cancels the first column and row of $\bfA$.  The equations that $\bfell$ must satisfy are given by
\begin{equation}
\label{eq:Cholesky}
\begin{aligned}
\ell_1\bar \ell_1 &= A_{11}; \\
\ell_2\bar \ell_1 &= A_{21}; \\
\ell_3\bar \ell_1 &= A_{31}; \\
\ldots &
\end{aligned}
\end{equation}
where $\bar x$ denotes the conjugate of $x$. (Conventionally, if $\bfA \in \Real^{K\times K}$, $\bfell \in \Real^K$ so conjugation is the identity, but for reasons that will become clear in the next section we will use notation that remains valid for complex numbers.) After solving these equations, we take
\begin{equation}
  \label{eq:schur}
  \bfA_2 \leftarrow \bfA_2 - \tilde\bfell\tilde\bfell^*,
\end{equation}
where $\tilde\bfell$ is $\bfell$ with the first entry omitted, and the process continues with the next row/column.  This results in a lower-triangular $\bfL$ created by assembling the vectors $\bfell$.

The first of \eqs{Cholesky} requires that $A_{11} \ge 0$.
Moreover, if $A_{11} = 0$ but $A_{21} \neq 0$, the second equation cannot be solved for any finite $A_{21}$.

These requirements are relaxed by the $LDL^T$ factorization, constructed by finding $d \bfell\bfell^*$ such that
\begin{equation}
\label{eq:LDLT}
\begin{aligned}
d \ell_1\bar\ell_1 &= A_{11}; \\
d \ell_2\bar\ell_1 &= A_{21}; \\
d \ell_3\bar\ell_1 &= A_{31}; \\
\ldots &
\end{aligned}
\end{equation}
While this allows $A_{11}$ to be of either sign, this does nothing to correct the problems that ensue from $A_{11} = 0$, because the first equation implies that $d\ell_1 = 0$ and this same product appears in each successive equation.  Symmetric pivoting strategies can be used to shift nonzero diagonals up to the first entry, but a matrix such as
\begin{equation}
\bfA = \begin{bmatrix}
0 & 1 \\
1 & 0
\end{bmatrix},
\end{equation}
even though it is non-singular, lacks an $LDL^T$ decomposition.

\section{The hemiplex numbers}

To solve these apparently-unsolvable equations, we ``invent'' a type of number with the necessary properties. (In reality, it is a rediscovery of one of the Jordan algebras\cite{jordan1933}, see Appendix~\ref{sec:Jordan}; this connection motivates the symbol $\mathbb{J}$.) We define a hemiplex number $z \in \mathbb{J}$ as one which may be written
\begin{equation}
\label{eq:hemiplex-number}
z = a + m \mu + n \nu,
\end{equation}
where $a, m, n \in \Real$. Any such number with $a = 0$ is called ``pure-hemi'' and denoted $z \in \mathbb{J}_0$. Addition and scalar multiplication of hemiplex numbers
follow the ordinary vector-space rules.  Multiplication is defined by giving the special numbers $\mu$ and $\nu$ the following algebraic properties:
\begin{equation}
\label{eq:hemiplex-multiply}
\mu^2 = 0, \qquad \nu^2 = 0, \qquad \mu\nu = \nu\mu = -\frac{1}{2}.
\end{equation}
As a consequence, both addition and multiplication are commutative.  Hemiplexes with either $m=0$ or $n=0$ are equivalent to dual numbers, but $\mu\nu = -1/2$ ensures that these numbers differ from the algebra of dual numbers with two generators: the $\mu$, $\nu$ terms are not infinitesimal.
We define the conjugate of a hemiplex number $z = a + m \mu + n \nu$ as
\begin{equation}
\label{eq:hemiplex-conjugate}
\bar z = a - m \mu - n \nu.
\end{equation}
This is the natural definition of conjugation:
\begin{lemma}
Let numbers $z \in \mathbb{J}$ be defined as \eq{hemiplex-number} with multiplication rules \eq{hemiplex-multiply}. \eq{hemiplex-conjugate} is the only linear map $z \rightarrow \bar z$ that acts as the identity operator on the reals, $\bar 1 = 1$, and satisfies $z \bar z \in \Real \: \forall \, z\in\mathbb{J}$.
\begin{proof}
   Let
   \begin{equation}
   \begin{aligned}
       \bar \mu &= \alpha + \beta\mu + \gamma\nu; \\
       \bar \nu &= \alpha' + \beta'\mu + \gamma'\nu
   \end{aligned}
   \end{equation}
   Then $\bar z = A + M\mu + N\nu$ with
   \begin{equation}
   \begin{aligned}
       A &= a + m\alpha + n\alpha'; \\
       M &= m\beta + n\beta'; \\
       N &= m\gamma + n\gamma'.
   \end{aligned}
   \end{equation}
   From this, requiring that
   \begin{equation}
       z \bar z = aA - \frac{mN + nM}{2} + (aM+mA)\mu + (aN+nA)\nu
   \end{equation}
   be real-valued requires
   \begin{equation}
       \begin{aligned}
           0 &= aM+mA = am(1+\beta) + an\beta' + m^2\alpha + mn\alpha'; \\
           0 &= aN+nA = an(1+\gamma') + am\gamma + nm\alpha + n^2\alpha'.
       \end{aligned}
   \end{equation}
   Requiring that these hold for all choices of $a$, $m$, and $n$ implies $\beta = \gamma' = -1$ and $\alpha = \alpha' = \beta' = \gamma = 0$.
\end{proof}
\end{lemma}

Because $|z|^2 = z \bar z = a^2 + mn$,  $|z|^2 = 0$ does not imply that $z$ is zero; any hemiplex number satisfying $a^2 + mn = 0$ with at least one nonzero among $a$, $m$, and $n$ is a divisor of zero without being zero. The hemiplexes are therefore not a real division algebra, which is not surprising given that there are only three such algebras (reals, complexes, and quaternions)\cite{frobenius1878lineare}.

Most hemiplexes $z\in\mathbb{J}$ have a unique multiplicative inverse: $1 = zw$ for $w = r + p\mu + q\nu$ requires zeroing the $\mu$ and $\nu$ coefficients of the product,
\begin{equation}
\label{eq:zeromunu}
\begin{aligned}
    ap + mr &= 0\\
    aq + nr &= 0.
\end{aligned}
\end{equation}
When $a \neq 0$ we may solve these for $p$ and $q$ in terms of $r$. Requiring the real-valued term to be one results in
\begin{equation}
r = a/|z|^2, \, p = -m/|z|^2, \, q = -n/|z|^2
\end{equation}
which is valid if $|z|^2 \neq 0$. Thus, we have
\begin{equation}
\label{eq:inverse}
    z^{-1} = \frac{\bar z}{|z|^2}
\end{equation}
just as with complex numbers. Note that $z^{-1}$ has the same $\mu/\nu$ ratio as $z$:
\begin{equation}
    \frac{-m/|z|^2}{-n/|z|^2} = \frac{m}{n},
\end{equation}
and $p/r = -m/a$, $q/r = -n/a$ so these ratios are also preserved except for a sign-flip. 

\eqs{zeromunu} can be satisfied trivially if $a = r = 0$, implying that any pure-hemi number $z_0 = m\mu + n\nu \in \mathbb{J}_0$ with $m \neq 0$ and/or $n\neq 0$ possesses an infinite number of pure-hemi multiplicative inverses $w$:
\begin{equation}
    1 = z_0w = wz_0 \textrm{ for } z_0 = m\mu + n\nu, w = p\mu + q\nu \qquad \Rightarrow \qquad mq + np = -2,
\end{equation}
which has a one-dimensional family of solutions.
%While division is therefore not unique, we define \emph{symmetric division} of $c \in \Real$ by a pure-hemi number $z = m\mu + n\nu$ as
% \begin{equation}
% \label{eq:symmetric_division}
% \frac{c}{z} = -\frac{c}{n}\mu - \frac{c}{m}\nu.
% \end{equation}
%Note that symmetric division preserves the $m/n$ ratio when neither $m$ nor $n$ is zero.  
For pure-hemis with $mn \neq 0$, we may adopt \eq{inverse} as a convention, thus defining \emph{symmetric division}, while recognizing that there are other inverses available. 

For the purposes of resolving the conundrum of section~\ref{sec:Cholesky}, the most salient property of the hemiplexes is that the equations
\begin{equation}
\label{eq:singularities}
\begin{aligned}
|\ell_1|^2 &= 0; \\
\ell_2\bar \ell_1  &= c, \qquad c \in \Real, \, c\neq 0
\end{aligned}
\end{equation}
are solved by
\begin{equation}
  \label{eq:handling-zeros}
\begin{aligned}
\ell_1 &= m\mu \\
\ell_2 &= p\mu + q\nu
\end{aligned}
\end{equation}
for arbitrary $p\in\Real$ and any $m$, $q\in\Real$ satisfying $mq = 2c$.  ($\mu$ and $\nu$ are symmetric, so of course there is a second family of solutions with $\ell_1 = n\nu$.) Among the real numbers, attempting to solve these equations requires $\ell_1 \rightarrow 0$, $\ell_2\rightarrow\infty$.  Consequently, hemiplexes allow one to replace expressions that involve singularities with operations on finite, well-scaled numbers. While hemiplexes also permit solution of $|\ell_1|^2 = A_{11}$ for $A_{11} < 0$, which in principle makes $LDL^T$ factorization unnecessary, numerical stability may encourage us to avoid cancellation in computing $a^2 + mn$, and this might argue in favor of solving such cases by introducing a diagonal that stores the signs of the diagonals as they are encountered. %This may be more convenient and numerically-stable than introducing a small nonzero real element $\epsilon$ and studying the properties of the system as $\epsilon\rightarrow 0$.

Perhaps the most unusual property of the hemiplexes is that multiplication is non-associative:  $(\mu\nu)\nu = -\nu/2$ but $\mu(\nu\nu) = 0$.  The necessity of abandoning associativity follows directly from \eqs{singularities}:
\begin{lemma}
Let $\cal A$ be a commutative algebra equipped with a homomorphism $x \rightarrow \bar x$. Let $x, y \in {\cal A}$ satisfy
\begin{subequations}
    \begin{align*}
        x\bar x &= 0; \\
        y\bar x &= c
    \end{align*}
\end{subequations}
for $c\in\Real$ with $c \neq 0$.
Then multiplication within ${\cal A}$ cannot be associative.
\end{lemma}
\begin{proof}
Right-multiply the second equation by $x$, resulting in $(y \bar x)x = cx$.  If multiplication were associative, we could regroup this as $y (\bar xx) = cx$.  By commutivity and the first equation, $\bar xx = x\bar x = 0$, implying that $c=0$ which conflicts with $c\neq 0$.
\end{proof}

For numerical work, the lack of associativity may be less of a problem than it may initially seem: if all multiplications are performed as they are needed, there is no need to keep track of grouping.

To gain more insight about whether hemiplex numbers might be used to compute stable Cholesky factorizations for arbitrary symmetric matrices, consider a perturbed variant of  \eqs{singularities},
\begin{equation}
\begin{aligned}
|\ell_1|^2 &= \epsilon; \\
\ell_2\bar \ell_1  &= 1
\end{aligned}
\end{equation}
for small $|\epsilon|$. Given $\epsilon \neq 0$, the first equation implies that $\bar \ell_1$ possesses an inverse of the form in \eq{inverse}, so we might be tempted to write
\begin{equation}
    \ell_2 = \bar\ell_1^{-1} = \frac{\ell_1}{|\ell_1|^2} = \frac{\ell_1}{\epsilon}.
\end{equation}
Since we also have $|\ell_1|^2 = \epsilon$, the best scaling we can obtain for $\ell_2$ is that it blows up as $1/\sqrt{|\epsilon|}$. This would not result in a stable algorithm. But hemiplexes offer a potential escape: if we make $\ell_1$ pure-hemi, we are free to make a different choice from \eq{inverse} when calculating $\ell_2$ from $\ell_1$. Thus, if we are to obtain fundamentally new capabilities from unconventional algebras, we need to exploit circumstances in which numbers have more than one multiplicative inverse, and for the hemiplexes that means restricting ourselves to factorizations that are pure-hemi.
% One stable factorization inverts the $\mu$/$\nu$ ratio
% \begin{equation}
% \begin{aligned}
% \ell_1 &= (\epsilon/2)\mu + \nu \\
% \ell_2 &= \frac{\mu + (\epsilon/2) \nu}{\sqrt{1 + \epsilon^2/4}}.
% \end{aligned}
% \end{equation}
% The limit of this solution as $\epsilon \rightarrow 0$ is consistent with \eqs{handling-zeros}.

% As will be seen, the tradeoff in stabilizing the factorization in this manner is that forward substitution 

\section{Example factorizations}
\label{sec:examples}

The Cholesky factorization over pure-hemi numbers is not unique. For example, the $2\times 2$ identity has the following factorizations:
\begin{align}
  \begin{bmatrix} 1 & 0 \\ 0 & 1 \end{bmatrix} &= \bfL \bfL^* \\
  \textrm{with } \bfL &= \begin{bmatrix} \mu+\nu & 0 \\ 0 & \mu+\nu \end{bmatrix} \textrm{ or } \begin{bmatrix} 5\mu+\frac{1}{5}\nu & 0 \\ 0 & \frac{1}{3}\mu+3\nu \end{bmatrix} \textrm{ or } \begin{bmatrix} -2\mu-\frac{1}{2}\nu & 0 \\ 0 & \frac{1}{\epsilon}\mu+\epsilon\nu \end{bmatrix}.
\end{align}

The following $\bfL\bfD\bfL^T$ factorization:
\begin{equation}
  \label{eq:badldlt}
  \begin{bmatrix} -\epsilon & 1 \\ 1 & 1 \end{bmatrix} = \begin{bmatrix} 1 & 0 \\ -\frac{1}{\epsilon} & 1 \end{bmatrix} \begin{bmatrix} -\epsilon & 0 \\ 0 & 1+\frac{1}{\epsilon} \end{bmatrix} \begin{bmatrix} 1 & -\frac{1}{\epsilon} \\ 0 & 1 \end{bmatrix}
\end{equation}
is not numerically stable as $\epsilon\rightarrow 0$\cite{Van95}, demonstrating that $\bfL\bfD\bfL^T$ factorization over the reals is not stable for indefinite matrices.  Two of the possible pure-hemi Cholesky factorizations are:
\begin{align}
  \bfL &= \begin{bmatrix} \sqrt{\epsilon}(\mu-\nu) & 0 \\ \frac{2}{\sqrt{\epsilon}}\nu & \mu+\nu \end{bmatrix}; \\
  \bfL &= \begin{bmatrix} \mu - \epsilon\nu & 0 \\ \frac{2(-\epsilon\mu + \nu)}{1+\epsilon^2} & \sqrt{1+\frac{4\epsilon}{(1+\epsilon^2)^2}} (\mu+\nu) \end{bmatrix}. \label{eq:example2}
\end{align}
The former is not stable, but the latter converges to a finite limit,
\begin{equation}
  \bfL \rightarrow  \begin{bmatrix} \mu & 0 \\ 2\nu & \mu+\nu \end{bmatrix},
\end{equation}
demonstrating that the hemiplex factorization is capable of gracefully handling near-zeros along the diagonal.

% L = [(-ϵ*μ + ν)/sqrt(2) 0; sqrt(2)/(1+ϵ^2)*(μ - ϵ*ν)  sqrt(1+4ϵ/(1+ϵ^2)^2)*(μ+ν)/sqrt(2)]

The matrix
\begin{equation}
\bfA = \begin{bmatrix}
 0 & 1 \\
 1 & 0
\end{bmatrix}
\end{equation}
does not have an $LDL^T$ decomposition, but it does have Cholesky factorizations over the hemiplexes; one choice is
\begin{equation}
  \label{eq:example3}
  \bfL = \begin{bmatrix}
  \mu &   0 \\
  2\nu & \mu
\end{bmatrix}.
\end{equation}

\section{A concrete hemiplex Cholesky factorization}

Given that the Cholesky factorization over the hemiplexes is not unique, developing practical algorithms requires that one makes some choices.  The examples of section~\ref{sec:examples} show that our conventions may have consequences for numeric stability.  Without making any claim of optimality, this manuscript explores one particular choice which appears to have useful properties.  To motivate this factorization, it is helpful to look ahead to the process of using the Cholesky factorization for solving linear equations.

\subsection{The $\mu/\nu$ ratio during backward-substitution}

The expression
\begin{equation}
\bfL \bfL^* \bfx = \bfb
\end{equation}
is equivalent to the pair
\begin{equation}
\bfL \bfy = \bfb, \qquad \bfL^* \bfx = \bfy.
\end{equation}
We have to solve the forward-substitution equations
\begin{equation}
\label{eq:forward}
L_{ii}y_i = g_i = b_i - \sum_{j=1}^{i-1} L_{ij}y_j
\end{equation}
and the backward-substitution equations
\begin{equation}
\label{eq:backward}
\bar L_{ii}x_i = h_i = y_i - \sum_{j=i+1}^K \bar L_{ji}x_j,
\end{equation}
where $x_i$ must be real.

Because $b_i\in\Real$ and $L_{ij}$ is pure-hemi, $y_i$ is pure-hemi and the $g_i$ are real-valued.  In solving \eq{forward} we have to choose among infinitely many possible solutions; in contrast, \eq{backward} is of the form $(m_L\mu+n_L\nu)x_i = m_h\mu + n_h\nu$, for which a real-valued solution exists only if $\frac{m_h}{n_h} = \frac{m_L}{n_L}$. Stated differently, there are no choices to be made during backward-substitution, but the choices made during forward-substitution determine whether the backward-substitution equations are solvable.

The most straightforward option would be to keep the $\mu/\nu$ ratio constant within columns of $\bfL$, as this ensures that \eq{backward} will have a real-valued solution.  However, this choice corresponds to division rule \eq{inverse} and does not encompass solutions like \eq{handling-zeros}, which is the main motivation for considering the hemiplexes.

Perhaps the next most straightforward option is to keep the $\mu/\nu$ ratio constant for all $L_{ij}$ with $i > j$, so that the $\bar L_{ji}x_j$ sum in \eq{backward} has predictable $\mu/\nu$ ratio.  We will therefore adopt this constraint.  %This will place restrictions on our choice for the $\mu/\nu$ ratio of $y_i$ during forward substitution. % and will be addressed in section~\ref{sec:solving}.

\subsection{Calculating the factorization}
\label{sec:decomposition}

This particular factorization convention is motivated by the observation that handling a zero along a diagonal, \eq{handling-zeros}, most simply requires that we choose ``opposite'' non-zero coefficients for $\mu$ and $\nu$ along the diagonal and sub-diagonal, respectively. This generalizes the choices made in the example \eq{example2} to arbitrary symmetric matrices. % To include this choice as the limit of a more general strategy, define the ratio
% \begin{equation}
%   \label{eq:ratio-definition}
%   d_i = \min\left(\frac{A_{ii}}{\max_{j>i} |A_{ji}|}, \sgn{A_{ii}}\right)
% \end{equation}
% during factorization of the $i$th column (so that here $A_{ji}$ already incorporates modifications from previous columns, \eq{schur}). Evidently $|d_i| \le 1$, and of course $d_i = 0$ if $A_{ii} = 0$ and at least one $A_{ji}$ is nonzero.  (If a column is all-zeros, we can choose $L_{ii} = \nu$ and all $L_{ji} = 0$ and move on to the next column.)  
We let $-1 \le d_j \le 1$ be the $\nu/\mu$ ratio for $L_{jj}$, and the $\mu/\nu$ ratio (note the inversion) for $L_{ij}$ with $i > j$:
\begin{subequations}
  \label{eq:component-ratio}
\begin{align}
  L_{jj} &= L_{jj\mu}(\mu + d_j \nu); \\
  L_{ij} &= L_{ij\nu}(d_j \mu + \nu) \qquad i > j.
  % L_{jj\mu} &= d_j L_{jj\nu}; \\
  % L_{ij\nu} &= d_j L_{ij\mu} \qquad i > j.
\end{align}
\end{subequations}
For the moment we defer specifying how $d_j$ is chosen.
We choose $L_{jj\mu}$ and $L_{ij\nu}$ to be finite, and $L_{jj\mu}$ is always non-zero.  ($L_{ij\nu}$ is zero if $A_{ij}$ is, which will enable efficient sparse factorizations.)
Concretely, from \eq{Cholesky} we have
\begin{subequations}
\begin{align}
  A_{jj} &= L_{jj\mu}^2|\mu + d_j\nu|^2 = L_{jj\mu}^2 d_j; \\
  A_{ij} &= L_{jj\mu}L_{ij\nu}(-\mu - d_j\nu)(d_j\mu + \nu) = L_{jj\mu}L_{ij\nu} \frac{1 + d_j^2}{2},
\end{align}
\end{subequations}
and thus we choose
\begin{subequations}
  \label{eq:L-definition}
\begin{align}
  L_{jj\mu} &= \begin{cases}
      1 & A_{jj} = 0;\\
      \sqrt{\frac{A_{jj}}{d_j}} & \textrm{otherwise};
  \end{cases} \\
  L_{ij\nu} &= \frac{2A_{ij}}{L_{jj\mu}(1+d_j^2)} \qquad i > j.
\end{align}
\end{subequations}
The other component of each entry of $L_{ij}$ is calculated from \eqs{component-ratio}.  After the conclusion of one column, we apply \eq{schur} and continue to the next column. We note that $d_j$ must have the sign of $A_{jj}$, and continuity in $L_{jj\mu}$ would require $d_j \rightarrow 0$ as $A_{jj} \rightarrow 0$.

Note that this factorization can be computed almost in-place in terms of purely real quantities: because of \eq{component-ratio}, one component is predictable from the other if we know $d_j$. Consequently we can store the real values of \eqs{L-definition} in place of $\bfA$, but we require separate storage for $d_j$.  Storing $\bfd$ is a cost of order $K$ for a $K\times K$ matrix.  We define $\bfL_\Real$ to be this real-valued lower triangular matrix,
\begin{subequations}
\begin{align}
  (\bfL_\Real)_{jj} &= L_{jj\mu}; \\
  (\bfL_\Real)_{ij} &= L_{ij\nu} \qquad i > j.
\end{align}
\end{subequations}

% How should we choose $d_i$? While we require $d_i \rightarrow 0$ as $A_{ii} \rightarrow 0$, for a symmetric positive definite matrix we might want $d_i = 1$ as this results in symmetric division. A fundamental property of such matrices is that $A_{ii} A_{jj} \ge A_{ji}^2$ for all $j > i$. Thus we adopt
% \begin{equation}
%     d_i = \sgn(A_{ii}) \min\left(1, \left\lbrace\frac{|A_{ii} A_{jj}|}{A_{ji}^2}\right\rbrace_{j>i}\right).
% \end{equation}
% This criterion is invariant under scale transformations of $\bfA$.

\subsection{Constraints from forward- and backward-substitution}

Substituting these choices into \eq{backward} results in the equation
\begin{equation}
  \label{eq:backward2}
  -L_{ii\mu}(\mu + d_i\nu) x_i - \sum_{j>i} L_{ji\nu}(d_i\mu+\nu)x_j = y_{i\mu}\mu + y_{i\nu}\nu.
\end{equation}
Define the real-valued quantity
\begin{equation}
  \label{eq:s-definition}
  s_i = \sum_{j>i} L_{ji\nu}x_j.
\end{equation}
Then solving \eq{backward2} implies
\begin{subequations}
\begin{align}
  L_{ii\mu}x_i + d_i s_i &= -y_{i\mu}; \label{eq:yconstr1} \\
  L_{ii\mu}d_i x_i + s_i &= -y_{i\nu},
\end{align}
\end{subequations}
and thus
\begin{equation}
  \label{eq:ycomponent-mu}
  y_{i\nu} = d_i y_{i\mu} + (d_i^2 - 1)s_i.
\end{equation}
Substituting this constraint into \eq{forward} yields
\begin{align}
  b_i &= L_{ii\mu} (\mu + d_i \nu)\left[y_{i\mu}\mu + d_i y_{i\mu}\nu + (d_i^2-1)s_i\nu \right] + \sum_{j<i}L_{ij\nu}(d_j\mu+\nu)\left[y_{j\mu}\mu + d_j y_{j\mu}\nu + (d_j^2-1)s_j\nu\right] \\
  &= -\frac{L_{ii\mu}}{2}\left(2d_iy_{i\mu} + (d_i^2-1)s_i\right) - \frac{1}{2}\sum_{j<i}L_{ij\nu}\left[(1+d_j^2)y_{j\mu} + d_j(d_j^2-1)s_j\right]. %\\
  % &= -L_{ii\nu}(1-d_i)^2\left(- y_{i\nu} + (1+d_i)s_i\right) - \sum_{j\le i}(\bfL_\Real)_{ij}\left[(1+d_j^2)y_{j\nu} + d_j(1-d_j^2)s_j\right].
\end{align}
Substitution of \eqs{yconstr1} and \eqn{s-definition} then allows this to be expressed as an equation in $\bfx$. Unfortunately, this is no longer a triangular system, so solution does not seem straightforward.
There is, however, at least one exception: if $d_j \in \{0, \pm 1\}$, the coefficient of $s_j$ in the sum over $j < i$ vanishes. Therefore, in our implementation we choose $d_j = \sgn(A_{jj})$.
In this case, forward-substitution is
\begin{equation}
\label{eq:forwardsubst-zeropiv}
    2L_{ii\mu}d_i y_{i\mu} = -2b_i + L_{ii\mu}(1 - d_i^2)s_i - \sum_{j<i}L_{ij\nu}(1+d_j^2)y_{j\mu}.
\end{equation}
This is a triangular system, but it still involves the ``unknowns'' (at the time of solution) $\{s_i\}$.  However, if there are a limited number of cases where $d_i = 0$, this reduces the solution of a large linear system to a much smaller linear system, one for which we seek the non-zero values of $s_i$ for the cases where $d_i = 0$.

% Therefore, here we explore the following choices:
% \begin{equation}
% \label{eq:Lconstruction}
% \begin{cases}
% m_i = n_i = \frac{A_{i1}}{\sqrt{2A_{11}}} & \textrm{when } A_{11} > 0; \\
% m_1 = 1, \, m_{i>1} = 0, \, n_i = A_{i1} & \textrm{when } A_{11} = 0; \\
% m_i = -n_i = \frac{A_{i1}}{\sqrt{2|A_{11}|}} & \textrm{when } A_{11} < 0.
% \end{cases}
% \end{equation}

\subsection{Handling zero pivots}
%We note that zero-pivots, where $d_i = 0$, may or may not correspond to singularities of $\bfA$. Thus, it should not be surprising that a general treatment has some flavor of nullspace analysis. 
We may formalize the result of \eq{forwardsubst-zeropiv} as

\begin{theorem}
\label{thm:singular}
Let $S$ denote the set of indexes $i$ with $d_i = 0$, and $i_p$ the $p$th such $i\in S$. Up to a possible correction $\bfb\rightarrow\bfb+\Delta\bfb$ restricted to the indexes $i\in S$, the solution $\bfx$ may be expressed as $\bfx = \tilde\bfx + \bfX\bfalpha$, where $\bfalpha$ is a vector of length $|S|$, and $\tilde\bfx$ and the $K\times |S|$ matrix $\bfX$ are solved by reference to only the $\mu$-component of \eq{backward} from $\tilde\bfy$ satisfying \eq{forward} and $\bfY$ lying in the null space of $\bfL$ with $Y_{i_p q} = \delta_{pq}\mu$.
\end{theorem}
\begin{proof}
During forward-substitution, entries with $i\in S$ have $L_{ii} = \mu$, and \eq{forward} is solved by $y_i = -2g_i\nu$.  However, combining this with \eq{component-ratio}, we see that $h_i$ has only the $\nu$ component whereas $L_{ii}$ has only the $\mu$ component, making \eq{backward} unsolvable in general.  We can correct this by solving \eq{forward} with
\begin{equation}
\label{eq:alpha}
y_i = -2g_i\nu + \alpha_p\mu,
\end{equation}
where $\alpha_p$ is an unknown real value that will be constrained by requiring the solvability of \eq{backward}.  We suppose that this $i$ is the $p$th such element of $S$.   We have $|S|$ unknown real values $\alpha_p$, and
\begin{equation}
\bfy = \tilde\bfy + \bfY\bfalpha,
\end{equation}
where $\tilde\bfy$ is the solution to \eq{forward} with all $\alpha = 0$.  During forward-substitution, these satisfy the forward-substitution equations
\begin{align}
L_{ii} {\tilde y}_i &= b_i - \sum_{j=1}^{i-1}L_{ij}{\tilde y}_j \label{eq:y0fwd} \\
L_{ii} {\bfY}_i &= - \sum_{j=1}^{i-1}L_{ij}{\bfY}_j
\end{align}
where we choose symmetric division when $i \notin S$; when $i\in S$, we choose the $\nu$-only solution, except for the $p$th column of $\bfY$ where ${Y}_{ip} = \mu$.  (The previous elements in this column of $\bfY$ are all zero.) Together, these equations propagate our unknown $\alpha_p$ throughout $\bfy$.

During backward-substitution, the solution is expressed as $\bfx = \tilde\bfx + \bfX\bfalpha$, where
\begin{align}
\bar L_{ii} {\tilde x}_i &= \tilde h_i = {\tilde y}_i - \sum_{j=i+1}^K \bar L_{ji} {\tilde x}_j \label{eq:h0} \\
\bar L_{ii} {\bfX}_i &= \bfh_i = {\bfY}_i - \sum_{j=i+1}^K \bar L_{ji} {\bfX}_j.
\end{align}
We iterate these with just the $\mu$ component of $\bfL$, as (when these equations are solvable) it is redundant with the $\nu$ component.  For those $i\in S$, solvability imposes the additional constraint that the $\nu$ component of $h_i$ must vanish:
\begin{equation}
\label{eq:nuvanish}
(\tilde h_i)_\nu + (\bfh_i)_\nu\cdot\bfalpha = 0,
\end{equation}
where $(h)_\nu$ denotes the $\nu$-component of the hemiplex $h$. There is one such constraint for each element of $S$, resulting in an $|S|\times|S|$ matrix equation
\begin{equation}
\label{eq:Heq}
\bfH\bfalpha = -(\tilde\bfh_S)_\nu.
\end{equation}
If $\bfH$ itself is singular, this equation may not have a solution for the right hand side resulting from the provided $\bfb$.  In such cases, solving \eq{Heq} in a least-squares sense results in computation of both $\bfalpha$ and the residual $\Delta\bfh$.  From \eq{h0}, $({\tilde h}_i)_\nu = ({\tilde y}_i)_\nu$ plus an offset, and from \eq{y0fwd} $({\tilde y}_i)_\nu = -2b_i$ plus an offset. Hence the elements of $\Delta\bfh/2$ may be added to the $\{b_i\}$ with $i\in S$ to yield a feasible solution.
\end{proof}
It is worth noting that the $\bfalpha$ are equal to the negatives of the values of $\bfx$ at indexes $i\in S$, as from \eqs{backward}, \eqn{alpha}, and \eqn{nuvanish} we have $-\mu x_i = \alpha_p\mu + 0\nu$.

Furthermore, we have
\begin{corollary}
\label{thm:rank}
$\bfX$ and $(\bfY)_\mu$ have maximal rank.
\end{corollary}
\begin{proof}
If we extract the rows indexed by $S$ of $(\bfY)_\mu$, by construction we obtain the identity matrix.  Since the identity matrix has full rank, the columns of $(\bfY)_\mu$ (and therefore $\bfY$) must be independent.  Since $\bfX$ satisfies $(\bfL^*)_\mu\bfX = (\bfY)_\mu$, the columns of $\bfX$ must be independent.
\end{proof}
\begin{corollary}
\label{thm:ynull}
$\bfY$ spans the null space of $\bfL$ consistent with symmetric division.
\end{corollary}
\begin{proof}
For $i\notin S$, given $b_1, \ldots, b_i$ and $y_1, \ldots, y_{i-1}$, the symmetric-division solution for $y_i$ is unique by \eq{inverse}.  Consequently, the only freedom is in choosing the values of the solution at $i\in S$.

Suppose $i$ is the smallest index in $S$.  Then by uniqueness $y_1, \ldots, y_{i-1}$ must be zero, and hence the $\nu$-component of $y_i$ must also be zero by $b_i = 0$.  As a consequence, between $i$ and the next-smallest index in $S$, $y$ must be proportional to the first column of $\bfY$.

Suppose up to the $p$th index $i$ in $S$, the solution is a linear combination of the first $p-1$ columns of $\bfY$.  Adding an independent nonzero $\nu$-component at the $p$th $i$ would likewise violate $b_i = 0$; consequently, only a $\mu$-component may be added, and thus up to the next index in $p$, the solution must be in the span of the first $p$ columns of $\bfY$.
\end{proof}
\begin{corollary}
\label{thm:nullspace}
The null space of $\bfA$ lies within the span of $\bfX$.
\end{corollary}
\begin{proof}
For any $\bfv \in \mathbb{R}^K$, $\bfy = \bfL^* \bfv$ is consistent with symmetric division.  If $\bfv$ is in the null space of $\bfA$, then $\bfL\bfy$ is all-zero.  By Corollary~\ref{thm:ynull}, $\bfy$ must be in the span of $\bfY$.  Since $(\bfY)_\mu = (\bfL^*)_\mu\bfX$, $\bfv$ must lie within the span of $\bfX$.
\end{proof}

Algorithms~\ref{alg:forward}--\ref{alg:backward} implement these operations.

\begin{algorithm}
\caption{Forward-substitution. $\tilde\bfy$ is a vector of pure-hemi numbers with the necessary multiplication rules defined; $/$ means symmetric division.}% $L_{i,j}$ refers to the value extracted from the returned matrix from algorithm~\ref{alg:decomposition}, converted into a hemiplex number according to algorithm~\ref{alg:representation}.}
\label{alg:forward}
\begin{algorithmic}
  \For{$i\gets 1\textrm{ to }K$}
    \State $g\gets b_i$
    \For{$j\gets 1\textrm{ to }i-1$}
      \State $g\gets g-L_{ij}\tilde y_j$
    \EndFor
    \If{$d_i = 0$}
      \State $\tilde y_i\gets -2g\nu$
    \Else
      \State $\tilde y_i\gets g/L_{ii}$
    \EndIf
  \EndFor
\end{algorithmic}
\end{algorithm}

\begin{algorithm}
\caption{Forward-substitution to define the symmetric-division null space of $\bfL$. Here $\bfg$ is a vector of length $|S|$ (the number of $d_i$ equal to zero), and $\bfY$ is $K\times |S|$ and is initialized with all-zeros. $v_{1:k}$ denotes the first $k$ entries of vector $\bfv$; where needed for disambiguation, here a comma is used to separate the two indices of a matrix.}
\begin{algorithmic}
  \State $s\gets 0$ \Comment{The number of $d_i=0$ encountered so far}
  \For{$i\gets 1\textrm{ to }K$}
    \State $\bfg\gets 0$
    \For{$j\gets 1\textrm{ to }i-1$}
      \State $g_{1:s}\gets g_{1:s}-L_{ij} Y_{j,1:s}$
    \EndFor
    \If{$d_i = 0$}
      \State $Y_{i,1:s}\gets -2\bfg_{1:s}\nu$
      \State $s\gets s+1$
      \State $Y_{is}\gets \mu$
    \Else
      \State $Y_{i,1:s}\gets g_{1:s}/L_{ii}$
    \EndIf
  \EndFor
\end{algorithmic}
\end{algorithm}

\begin{algorithm}
\caption{Backward-substitution. Here $\bfx$, $\bfy$ refer to either $\tilde\bfx$, $\tilde\bfy$ or $\bfX$, $\bfY$. This is written assuming a single column of $\bfx$ and $\bfy$, but each operation may be iterated over columns. $\bfH$ refers to the vector/matrix for solving for the constraints, \eq{Heq}.}
\label{alg:backward}
\begin{algorithmic}
  \State $s\gets 0$ \Comment{The number of $d_i=0$ encountered so far}
  \For{$i\gets K\textrm{ to } 1$}
    \State $h\gets -y_i$
    \For{$j\gets i+1\textrm{ to }K$}
      \State $h\gets h- L_{ji}x_j$
    \EndFor
    \If{$d_i = 0$}
      \State $x_i\gets (h)_\mu$
      \State $s\gets s+1$
      \State $H_{s}\gets (h)_\nu$
    \Else
      \State $x_i\gets (h)_\mu/(L_{ii})_\mu$
    \EndIf
  \EndFor
\end{algorithmic}
\end{algorithm}

\section{Experiments and performance analysis}

A Julia implementation of the hemiplexes and their application to Cholesky factorization may be found at \url{https://github.com/timholy/HemiplexNumbers.jl} and \url{https://github.com/timholy/HemiplexFactorizations.jl}.  This code, and all pseudocode contained in this manuscript, are released under the permissive MIT license.

These packages implement the decomposition in section~\ref{sec:decomposition}, as well as a variant incorporating diagonal pivoting to place the largest-magnitude diagonal element in the first position.  For the unpivoted algorithm, computational performance was enhanced by implementing blocking (all experiments used a block size of 32 columns).  Both algorithms performed the subtraction of $\bfell\otimes\bfell$ via calls to the level-2 BLAS functions \texttt{xSYRK} and \texttt{xSYR2K}.  Note that the latter, necessary when the block's $d_i$ values are heterogenous, may introduce a penalty compared to a (hypothetical) \texttt{xSYRK} implementation allowing one to pass a vector of diagonals rather than a single scalar $\alpha$.

Comparisons were made using the Julia library functions for Cholesky (\texttt{cholfact}), Bunch-Kaufman (\texttt{bkfact}), and eigenvalue decomposition (\texttt{eigfact}), and solution via the \texttt{\textbackslash} operator.  These dispatch to standard LAPACK implementations of Cholesky (\texttt{xPOTRF} and \texttt{xPOTRS}), Bunch-Kaufman (\texttt{xSYTRF} and \texttt{xSYTRS}), and eigendecomposition (\texttt{xSYEVR}) algorithms.  We used both the traditional and the rook-pivoting (\texttt{xSYTRF\_ROOK}) variants of Bunch-Kaufman.  Comparison to eigenvalue decomposition was motivated by the observation that it represents a widely-used approach for truncating singular or nearly-singular directions when solving symmetric linear equations.  Run times were measured using \texttt{BenchmarkTools.jl}'s \texttt{@belapsed} macro. LAPACK/BLAS operations used 8 threads on an Intel i7-14700KF with 20 cores.  The hemiplex Cholesky factorization truncated diagonals at $10K\epsilon |\bfA|_\infty$, where $\epsilon$ is the machine precision and $|\bfA|_\infty$ is the maximum-magnitude element of $\bfA$; the solution via eigenfactorization was truncated similarly, but using the eigenvalues rather than entries of $\bfA$.

Test matrices included random matrices $\bfQ\bfLambda\bfQ^T$ constructed by the method of Stewart\cite{Stewart1980}, where the diagonal matrix $\bfLambda$ satisfied\cite{fang2011stability}
\begin{equation}
|\lambda_i| = \beta^i, i = 1,\ldots,r-1, \qquad \lambda_r = 1
\end{equation}
where $r$ is the rank, $\beta = \sigma^{1/(r-1)}$ for some $0 < \sigma \le 1$.  All other diagonals from $r$ to $K$, the size of the matrix, were zero.  We used $\sigma = 1, 10^{-3}, \ldots, 10^{-12}$.  In some experiments, the sign of $\lambda_i$ was chosen randomly or, for positive-definite matrices, as $+1$.  We also considered special matrices: Hadamard matrices, Toeplitz matrices, and the matrices studied by Druinsky \& Toledo\cite{Druinsky2011} in their analysis of stability of Bunch-Kaufman factorization, in both the singular and non-singular variants (with final element $A_{KK} = 0$).

\figref{factorization_performance} shows the factorization performance of the various algorithms as a function of matrix size.  Run time was similar to LAPACK's Bunch-Kaufman despite the fact that key elements of the hemiplex algorithm are implemented as single-threaded computations in Julia rather than optimized LAPACK implementations.  On positive-definite matrices, it is not as fast as the ordinary Cholesky algorithm; since the operations are almost mathematically identical, this difference seems likely to be just an implementation issue.  The pivoted algorithm, as implemented here, is slower because it does not use blocking; however, it is still far faster than eigenvalue decomposition.

\begin{figure}
\centering \includegraphics[width=10cm]{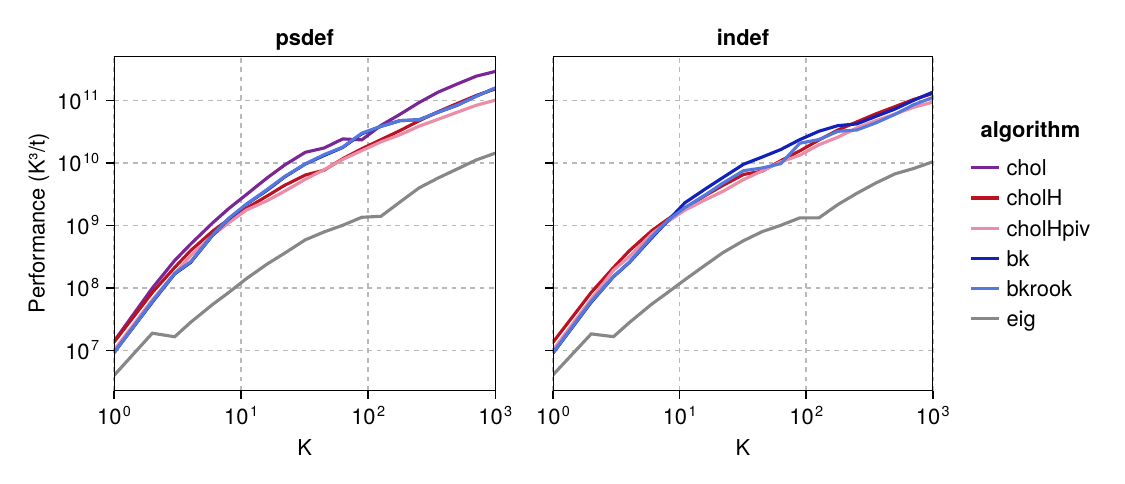}
\caption{\label{fig:factorization_performance}Factorization performance (higher is better) for positive-semidefinite (psdef) and indefinite (indef) $K\times K$ random matrices using different factorization algorithms. $t$ is the execution time of the factorization.}
\end{figure}

For solving equations, the right hand side $\bfb$ was generated from $\bfx_\star = [1,1,\ldots]$ and multiplying by $\bfA$.  For the hemiplex factorizations with singularities, the least-squares solution was generated.  The median accuracy, measured by $|\bfA\bfx - \bfb|$ (\texttt{norm} in Julia), is shown in \figref{accuracy1} using 30 random matrices for each condition.  For random matrices, one sees that the unpivoted hemiplex factorization was consistently among the least accurate, and that the pivoted hemiplex factorization was consistently among the most accurate.  For the random matrices, the rank as estimated from the pivoted hemiplex factorization was correct in every tested case.

\begin{figure}
\centering \includegraphics[width=\textwidth]{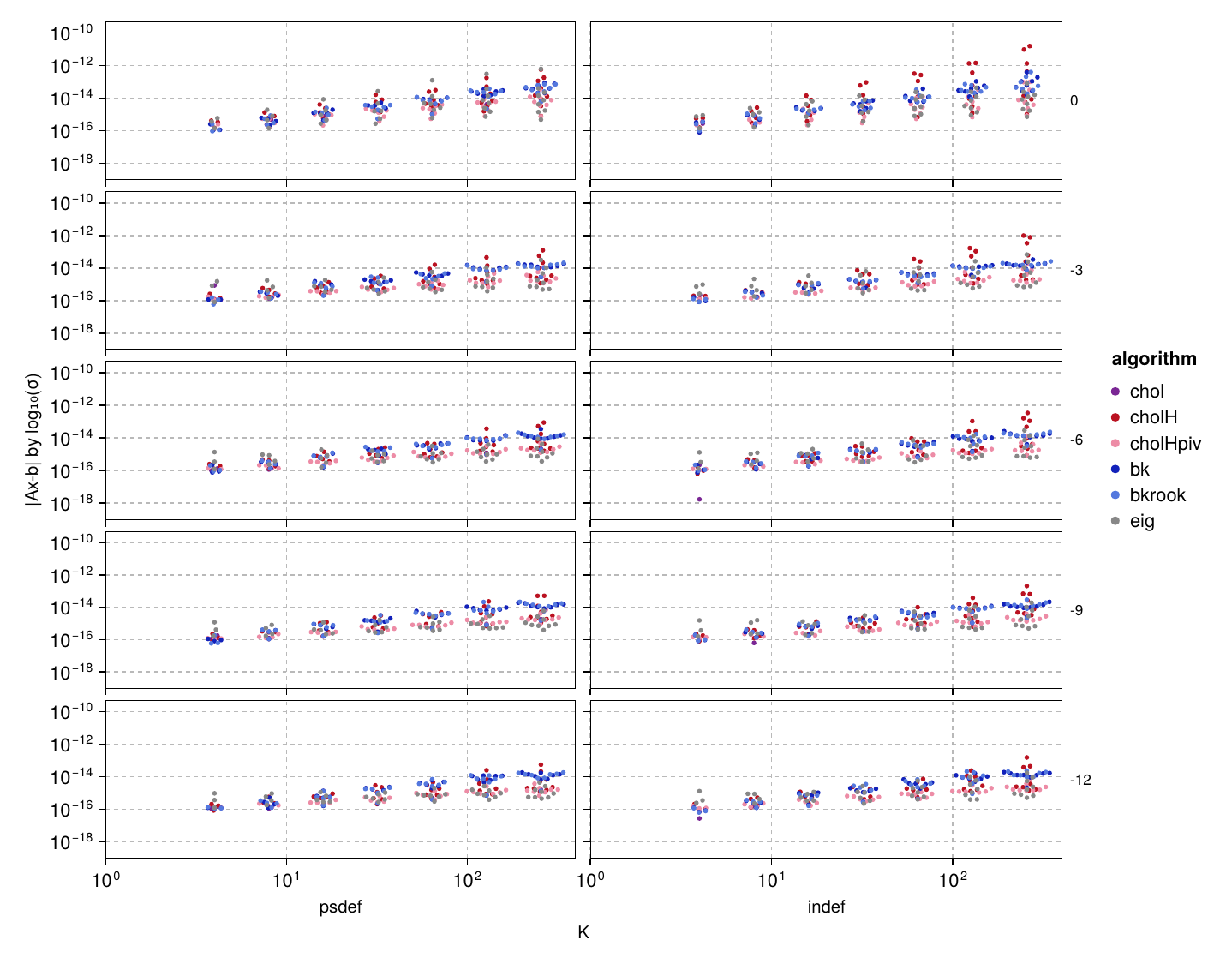}
\caption{\label{fig:accuracy1}Norm of residual in solution (lower is better) with different factorization algorithms for random positive-semidefinite (psdef) and indefinite (indef) matrices, as a function of $\log_{10}\sigma$.  Each condition was tested with ranks $r$ having all integer-powers $p$ of 2 up to $2^p = K$.  Points are offset along the $x$-axis to avoid overlap; in all cases $K$ was a power of 2 (maximum 256).}
\end{figure}

The results for specialized matrices are shown in \figref{accuracy2}. Notably, the larger Druinsky-Toledo matrices (both singular and nonsingular) caused both the traditional and rook-pivoting Bunch-Kaufman LAPACK factorizations to issue singularity errors. However, both hemiplex Cholesky factorizations succeeded, and for the pivoted algorithm the (pseudoinverse) solution was of reasonable accuracy.

\begin{figure}
\centering \includegraphics[width=8cm]{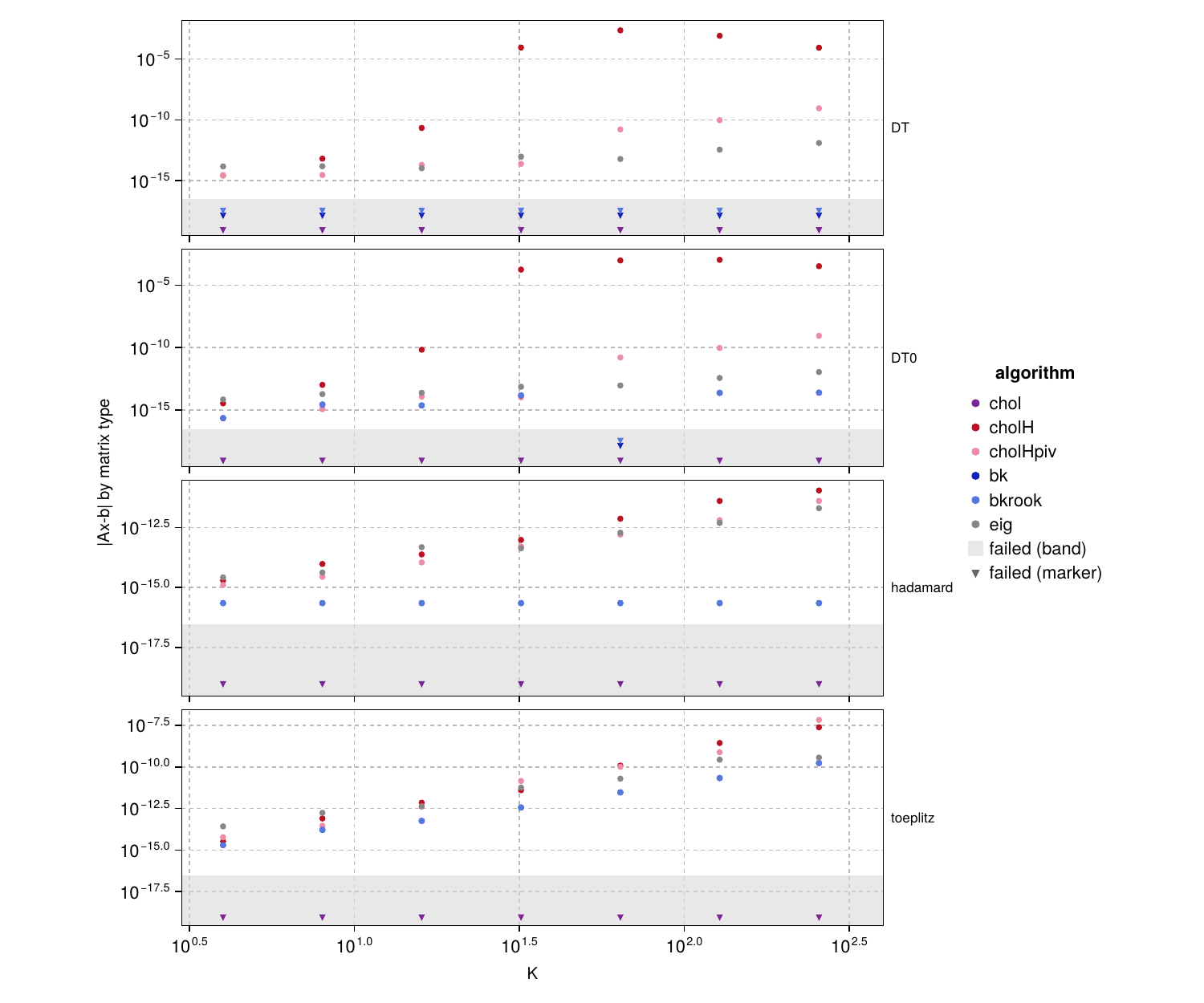}
\caption{\label{fig:accuracy2}Norm of residual in solution (lower is better) with different factorization algorithms for special matrices: Druinsky-Toledo (DT), nonsingular Druinsky-Toledo (DT0), Hadamard matrices, and Toeplitz matrices.  Points are offset along the $x$-axis to avoid overlap. For the purposes of display, machine-$\epsilon$ was added to each value, as Bunch-Kaufman factorization was exact (zero norm) for the Hadamard matrices.  Some algorithms (real-valued Cholesky, and both Bunch-Kaufman) failed to factor the matrix, in which case the corresponding dot is absent.}
\end{figure}

\section{Discussion}

Real numbers have been extended in many ways, including the complex numbers, split-complex (also called hyperbolic or duplex) numbers, dual numbers, quaternions, and Clifford algebras.  The hemiplexes obey a simple algebra, and are one of 26 non-isomorphic three-dimensional Jordan algebras over the reals\cite{Kashuba2016}. 
However, we have not been able to find any prior efforts that demonstrate their utility for practical problems.

The most unusual property of the hemiplexes is their non-associative multiplication.  In some contexts, this seems likely to present a challenge for analytical progress. For matrix factorization as a pair $\bfA = \bfB \bfC$, only two numbers appear in any product, and hence the non-associativity plays little obvious role.  For numeric computation, non-associativity seems to be less of a concern, because immediate computation implicitly imposes grouping without any extra cost.  The lack of a unique multiplicative inverse---a side effect of their non-associativity---plays a much more explicit role in this work, but unexpectedly we found that this can be an advantage: it allows one to naturally introduce what is effectively ``bookkeeping'' to be resolved later.

Corollary~\ref{thm:nullspace}, together with the ability to ``truncate'' some near-zero diagonals, endow the hemiplex Cholesky factorization with some properties of the QR and singular value decompositions:  it solves any symmetric matrix equation, near-singular directions may be discarded to prevent contamination of solutions by roundoff error, and least-squares solutions may be obtained by projecting out those components of $\bfx$ and $\bfb$ that are parallel to the null space of $\bfA$.  A very large matrix of nearly full rank results in a small system of equations for $\bfalpha$, so the extra cost of handling a matrix with a nonzero null space is due to effectively solving for several right-hand sides, an $O(K^2)$ operation typically much smaller than the $O(K^3)$ needed for factorization. There are, however, pathological cases: a symmetric tridiagonal matrix with all-zeros on the diagonal will lead to $K$ zero-pivots, in which case hemiplex factorization offers no advantage. In ``typical'' cases, the fact that the Cholesky decomposition can be calculated more efficiently than other factorizations suggests that this approach may have practical advantages.

An additional feature of the hemiplex factorization is that no pivoting is required for its existence.  For dense matrices, this is presumably of only minor benefit, as pivoting is $O(K^2)$; moreover, as shown here, pivoting can still improve the accuracy of solutions.  For sparse matrices, a major advantage of Cholesky factorization is its preservation of banded structure.  For such matrices, pivoting is usually designed by reference to structure but not to the values of the nonzero entries.  Consequently, particularly for sparse matrices, the fact that pivoting is not essential may be an additional benefit.

We close by highlighting one important open question. From the standpoint of numeric stability, it would be very desirable to allow arbitrary $-1 \le d_i \le 1$ so that \eq{L-definition} could be made continuous. The authors note that
\begin{equation}
    r_{kj} = \frac{|A_{jj} A_{kk}|}{A_{jk}^2}
\end{equation}
is a scale-invariant measure of the ``positive-definiteness gap'' of the $j, k$ principal $2\times 2$ positive-diagonal submatrix, and setting $d_j = \sgn{A_{jj}} \min(1, \max_{k > j, A_{jk} \neq 0} r_{kj})$ might be an attractive choice. This would result in solutions similar to \eq{example2}. However, implementing this would require new methods for making the necessary choices during forward substitution.  

\appendix
\section{Jordan Algebras}
\label{sec:Jordan}

It is well known that the algebra of complex numbers $a+bi$  is isomorphic to the algebra
of real $2\times 2$ matrices of the form
$$
\begin{pmatrix}
a & -b \\
b & a
\end{pmatrix}.
$$
Thanks to the popularity of automatic differentiation, dual numbers
introduced in 1873 by William Clifford
are becoming more familiar.  These are numbers of the form $a+b\epsilon$,
where $\epsilon^2=0$. These are also isomorphic to a subalgebra of
real $2\times 2$  matrices, namely
$$
\begin{pmatrix}
a & b \\
0  & a
\end{pmatrix}.
$$

One might ask if the hemiplexes can be realized as a subalgebra of 
$2\times 2$ real matrices. The answer is yes, but with
a matrix product redefined to make multiplication commutative.
Hemiplexes  correspond to  $2\times 2$ real matrices with equal diagonal,
with the commutative multiplication:
\begin{equation}
\label{multcomm}
 A \circ B \equiv (AB+BA)/2,
 \end{equation}
where ordinary matrix multiplication is on the right side of the definition.

The association is
\begin{equation}
a + m\mu + n\nu  \sim
\begin{pmatrix}
a & m \\
-n & a
\end{pmatrix}.
\end{equation}
Therefore $1$ corresponds to $I_2$ and $\mu$ and $\nu$ correspond to
$
\begin{pmatrix}
0 & 1 \\
0 & 0
\end{pmatrix}
$
and
$
\begin{pmatrix}
0  & 0 \\
-1 & 0
\end{pmatrix}
$
respectively.

The  algebra can also be characterized by
$e_1=1$, $e_2=\mu-\nu$, $e_3=\mu+\nu$, with the observation
that $e_2^2=-e_3^2=e_1$ and $e_1  e_i=e_i$ for $i=1,2,3$.

The hemiplexes are an example of a Jordan algebra\cite{jordan1933}. Jordan algebras are commutative, but not necessarily associative, but satisfy rather the weaker relationship
$$x^2(xy) =x(x^2 y).$$  They are named for the 20th century German  physicist
Pascual Jordan, not to be confused with the famous French mathematician Camille
Jordan.

There is a nice list of small dimensional Jordan algebras in \cite{Kashuba2016}.
The hemiplexes appear there as ${\cal J}_3 $ in the list of five  real  semi-simple Jordan
algebras of dimension 3.

\section*{Author contributions}
TEH designed the hemiplexes, derived the algorithms and proofs for Cholesky factorization, wrote the software, and performed the numerical experiments. AE discovered the connection to the Jordan algebras and identified the correct conjugation rule.

\section*{Acknowledgments}

The authors are grateful to Jiahao Chen for useful conversations, and Jack Poulson for advice regarding algorithm implementation and the Druinsky-Toledo test matrices.
We also thank Dav Vogan for discussions about Jordan Algebras.
The authors also acknowledge the Julia community for providing a programming language that makes it easy to play with unusual types of numbers, and for demonstrating in other contexts how specialized numbers can be practically useful.

\bibliographystyle{siamplain}
\bibliography{refs1}

\end{document}